\documentclass[reqno]{amsart}

\usepackage{amssymb,amsmath,amsfonts,amsthm}
\usepackage[all]{xy}
\usepackage{enumerate}
\usepackage{mathrsfs}
\usepackage{graphicx}

\usepackage{listings}

\theoremstyle{plain}
\newtheorem{thm}{Theorem}
\newtheorem{cor}[thm]{Corollary}
\newtheorem{lem}[thm]{Lemma}

\newtheorem{prop}[thm]{Proposition}

\theoremstyle{definition}
\newtheorem{defn}[thm]{Definition}

\theoremstyle{remark}
\newtheorem{rem}[thm]{Remark}
\newtheorem{exmp}[thm]{Example}

\newcommand\dto{\dashrightarrow}

\newcommand\lto{\longrightarrow}

\def\leq{\leqslant}
\def\geq{\geqslant}

\newcommand\ZZ{\mathbb Z}

\newcommand\PP{\mathbb P}
\newcommand\Sc{\mathcal S}
\newcommand\Rc{\mathcal R}
\newcommand\Hc{\mathcal H}
\newcommand\Zc{\mathcal Z}
\newcommand\Cc{\mathcal C}
\newcommand\Oc{\mathcal O}
\newcommand\Fitt{\mathfrak F}
\newcommand\Scc{\mathscr S}

\newcommand\mm{\mathfrak m}

\def\ud{{\underline{d}}}
\def\unu{{\underline{\nu}}}
\def\Rf{{\mathfrak{R}}}
\def\ue{{\underline{e}}}

\DeclareMathOperator\Rees{Rees}

\DeclareMathOperator\Sym{Sym}

\DeclareMathOperator\ann{ann}

\DeclareMathOperator\Spec{Spec}

\DeclareMathOperator\Proj{Proj}

\DeclareMathOperator\reg{reg}
\DeclareMathOperator\indeg{indeg}

\def\pp{{\mathfrak{p}}}

\def\CC{{\mathbb{C}}}
\def\corank{{\mathrm{corank}}}

\title[Fitting ideals and multiple-points of surface parameterizations]{Fitting ideals and multiple-points of surface parameterizations}

\author{Nicol\'as Botbol}
\thanks{Nicol\'as Botbol was partially supported by Marie-Curie Network ``SAGA" FP7 contract
PITN-GA-2008-214584, EU}
\address{Departamento de Matem\'atica, 
FCEN, Universidad de Buenos Aires, Argentina.}
\email{nbotbol@dm.uba.ar}

\author{Laurent Bus\'e}
\address{INRIA Sophia Antipolis - M\'editerran\'ee, Galaad team, 2004 route des Lucioles, B.P.~93, F-06902, Sophia Antipolis France.}
\email{Laurent.Buse@inria.fr}

\author{Marc Chardin}
\address{Institut Math\'ematique de Jussieu et 
Universit\'e Pierre et Marie Curie, Bo\^ite 247,
4 place Jussieu, F-75252 Paris CEDEX 05, France.}
\email{chardin@math.jussieu.fr}

\date{\today}

\begin{document}

 \begin{abstract}
 Given a birational parameterization $\phi$ of an algebraic surface $\Scc$ in the projective space $\PP^3$, the purpose of this paper is to investigate the sets of points on $\Scc$ whose preimage consists in $k$ or more points, counting multiplicities. They are described explicitly in terms of Fitting ideals of some graded parts of the symmetric algebra associated to the parameterization $\phi$.
 \end{abstract}

\maketitle

\section{Introduction}

Parameterized algebraic surfaces are ubiquitous in geometric modeling because they are used to describe the boundary of 3-dimensional shapes. To manipulate them, it is very useful to have an implicit representation in addition to their given parametric representation. Indeed, a parametric representation is for instance well adapted for drawing or sampling whereas an implicit representation allows significant improvements in intersection problems that are fundamental operations appearing in geometric processing for visualization, analysis and manufacturing of geometric models. Thus, there exists a rich literature on the change of representation from a parametric to an implicit representation under the classical form of a polynomial implicit equation. Although this problem can always be solved in principle, for instance via Gr\"obner basis computations, its practical implementation is not enough efficient to be useful in practical applications in geometric modeling for general parameterized surfaces.   

In order to overcome this difficulty, alternative implicit representations of parameterized surfaces under the form of a matrix have been considered. The first family of such representations comes from the resultant theory that produces a non-singular matrix whose determinant yields an implicit equation from a given surface parameterization. But the main advantage is also the main drawback of these resultant matrices: since they are universal with respect to the coefficients of the given surface parameterization, they are very easy to build in practice, but they are also very sensitive to the presence of base points. As a consequence, a particular resultant matrix has to be designed for each given particular class of parameterized surfaces. 
A second family of implicit matrix representations is based on the syzygies of the coordinates of a surface parameterization. Initiated by the geometric modeling community \cite{SC95}, these matrices have been deeply explored in a series of papers (see \cite{BuJo03,BC05,BDD08,Bot10} and the references therein). Compared to the resultant matrices, they are still very easy to build, although not universal, but their sensitivity to the presence of base points is much weaker. However, these matrices are in general singular matrices and the recovering of the classical implicit equation from them is more involved. Therefore, to be useful these matrices have to be seen as implicit representations on their own, without relying on the more classical implicit equation. In this spirit, the use of these singular matrix representations has recently been explored for addressing the curve/surface and surface/surface intersection problems  \cite{BL12,LBBM09}. 

As a continuation, the purpose of this paper is to investigate the self-intersection locus of a surface parameterization through its matrix representations. 
More precisely,  let $\phi$ be a parameterization from $\PP^2$ to $\PP^3$ of a rational algebraic surface $\Sc$ and let $M(\phi)$ be one of its matrix representations. The main result of this paper is that the drop of rank of $M(\phi)$ at a given point $P\in\PP^3$ is in relation with the fiber of the graph of $\phi$ over $P$. Thus, the Fitting ideals attached to $M(\phi)$ provide a filtration of the surface which is in correspondence with the degree and the dimension of the fibers of the graph of the parameterization $\phi$. 

It turns out that this kind of results have already been investigated for different purposes in the field of intersection theory under the name of multiple-point formulas. Given a finite map of schemes $\varphi:X\rightarrow Y$ of codimension one, its source and target double-point cycles have been extensively studied (see e.g.~\cite{KLU96,KLU92,Pie78,Tei77}). Moreover, in \cite{MP89} Fitting ideals are used to give a scheme structure to the multiple-point loci. Therefore, in the particular case where $\phi:\PP^2\rightarrow \PP^3$ has no base point (hence is a finite map), our results are partially contained in the above literature. Thus, the main contribution of this paper is to extend these previous works to the case where $\phi$ is not necessarily a finite map under ubiquitous conditions. The theory in this article is thus applicable to most of the parameterizations that appear in applications.

\medskip

In what follows, we will first briefly overview matrix representations and define some Fitting ideals attached to them in Section \ref{sec:def}. The main result of this paper is then proved in Section \ref{sec:main}. Section \ref{sec:comp} is devoted to the computational aspects of our results. In Section \ref{sec:link} we will discuss on the link with multiple-points formulas developed in the field of intersection theory. Finally, we will treat the case of parameterizations $\phi$ whose source is $\PP^1\times \PP^1$ instead of $\PP^2$. This type of parameterizations being widely used in geometric modeling (see for instance \cite{SSV12} for a detailed study of a special case).

\section{Fitting ideals associated to surface parameterizations}\label{sec:def}

Let $k$ be a field and $\phi: \PP^2_k \dto \PP^3_k$ be a rational map given by four homogeneous polynomials $f_0,f_1,f_2,f_3$ of degree $d$. We will denote by $\Scc$ the closure of its image. Set $s:=(s_0,s_1,s_2)$, $x:=(x_0,x_1,x_2,x_3)$, and let $S:=k[s]$ and $R:=k[x]$ be the polynomial rings defining $\PP^2_k$ and $\PP^3_k$ respectively. Finally, define the ideal $I=(f_0,f_1,f_2,f_3)\subset S_d$. In the sequel, we will assume that:
\begin{equation}\label{H}\tag{$\Hc$}
\left.
\begin{array}{l}
- \  \Scc \textrm{ is a surface in } \PP^3_k, \\
- \   \phi \textrm{ is a birational map onto } \Sc, \\
- \   V(I) \textrm{ is finite and locally a complete intersection.}
\end{array}\right\}
\end{equation}

\medskip

\subsection{Matrix representations of $\phi$}
The approximation complex of cycles associated to the sequence of polynomials $f_0,f_1,f_2,f_3$ is of the form (see e.g.~\cite[\S4]{BuJo03} for an introduction to these complexes in this context)
\begin{equation}\label{eqZComplex}
 \Zc_\bullet: \qquad 0\to \Zc_3 \to \Zc_2 \to \Zc_1 \stackrel{M(\phi)}{\lto} \Zc_0 \to 0.
\end{equation}
The map $M(\phi)$ sends a 4-tuple $(g_0,g_1,g_2,g_3)$ to the form $\sum_{i=0}^{3}g_ix_i$, with the condition that $(g_0,g_1,g_2,g_3) \in \Zc_1$ if and only if $\sum_{i=0}^{3}g_if_i=0$. The complex $\Zc_\bullet$ inherits a grading from the canonical grading of the polynomial ring $S$. Thus, the notation $M(\phi)_\nu$ stands for the matrix in $(\Zc_\bullet)_\nu$ that corresponds to the degree $\nu$ part of $M(\phi)$.

Under our hypothesis,
 the symmetric algebra $\Sc_I:=\Sym_S(I)$ of the ideal $I\subset S$  is projectively isomorphic to the Rees algebra $\Rc_I:=\Rees_S(I)$. In other words, the graph $\Gamma \subset \PP^2_k\times \PP^3_k$ of the parameterization $\phi$ is scheme defined by $\Sc_I$. As a consequence, the approximation complex of cycles $\Zc_\bullet$ can be used to determine an implicit representation of the closed image of $\phi$, i.e.~the surface $\Scc$ (recall that $H_0(\Zc_\bullet)=\Sym_S(I)$). 

\begin{defn} Under the assumptions \eqref{H}, for all integer $\nu\geq \nu_0$ the matrix $M(\phi)_\nu$ of the graded map of free $R$-modules $(\Zc_1)_\nu\rightarrow (\Zc_0)_\nu$ is called a \emph{matrix representation} of $\phi$.	
\end{defn}

For the moment, the integer $\nu_0$ can be taken equal to $2d-2$, but we will come back to its definition in the next section. Here are the main features of the collection of matrix representations $M(\phi)_\nu$ of $\phi$, for $\nu\geq \nu_0$:
\begin{itemize}
	\item Its entries are linear forms in $R=k[x_0,\ldots,x_3]$,
	\item it has $\binom{\nu+2}{2}$ rows and at least as many columns as rows,
	\item it has maximal rank $\binom{\nu+2}{2}$ at the generic point of $\PP^3_k$, that is to say 
	$$\mathrm{rank} \, M(\phi)_\nu \otimes_{R}\mathrm{Frac}(R) = \binom{\nu+2}{2},$$
	\item when specializing $M(\phi)_\nu$ at a point $P\in \PP^3$, its rank drops if and only if $P\in \Sc$.
\end{itemize}

\subsection{Fitting ideals associated to $\phi$}
From the properties of matrix representations of $\phi$, for all $\nu\geq \nu_0$ the support of the Fitting ideal $\Fitt^0((\Sc_I)_\nu)$ is  $\Scc$ and hence provides a scheme structure for the closure of the image of $\phi$. Following \cite{Tei77} (and \cite[V.1.3]{EH00}), we call it the \emph{Fitting image} of $\phi$.

\begin{rem}
	Observe that by definition, the ideals $\Fitt^0((\Sc_I)_\nu)$ depend on the integer $\nu$ (it is generated in degree $\nu$) whereas $\ann_R((\Sc_I)_\nu)=\ker(\phi^\sharp)$ for all $\nu\geq \nu_0$ (see \cite[Proposition 5.1]{BuJo03} and \cite[Theorem 4.1]{BC05}), where $\phi^\sharp:R\to S$ is the map of rings induced by $\phi$ (it sends each $x_i$ to $f_i(s)$ for all $i=0,\ldots,3$) and $\ker(\phi^\sharp)$ is the usual ideal definition of the image of $\phi$.	
\end{rem}

In this paper, we will push further the study of $\Fitt^0((\Sc_I)_\nu)$ by looking at the other Fitting ideals $\Fitt^i((\Sc_I)_\nu)$, $i>0$, since they provide a natural stratification~:
$$ \Fitt^0((\Sc_I)_\nu) \subset \Fitt^1((\Sc_I)_\nu) \subset \Fitt^2((\Sc_I)_\nu) \subset \cdots \subset  \Fitt^{\binom{\nu+2}{2}}((\Sc_I)_\nu)=R.$$
As we will see, these Fitting ideals are closely related to the geometric properties of the parameterization $\phi$.  For simplicity, the Fitting ideals $\Fitt^i((\Sc_I)_\nu)$ will be denoted $\Fitt^i_\nu(\phi)$. We recall that  $\Fitt^i_\nu(\phi) \subset R$ is generated by all the minors of size $\binom{\nu+2}{2}-i$ of any matrix representation $M(\phi)_\nu$.

\begin{exmp} Consider the following parameterization of the sphere 
\begin{eqnarray*}
	\phi : \PP^2_\CC & \dto & \PP^3_\CC \\
	 (s_0:s_1:s_2) & \mapsto & (s_0^2+s_1^2+s_2^2: 2s_0s_2:2s_0s_1:s_0^2-s_1^2-s_2^2).
\end{eqnarray*}	
Its matrix representations $M(\phi)_\nu$ have the expected properties for all $\nu\geq 1$ (see the next section). The computation of the smallest such matrix yields 
$$M(\phi)_1=
\left(\begin{array}{cccc}
 0 &   X_1 &      X_2 &     -X_0+X_3 \\
	X_1 & 0   &     -X_0-X_3 & X_2 \\
	-X_2 &  -X_0-X_3 & 0       & X_1  
\end{array}\right).$$
Here is a Macaulay2 \cite{M2} code computing this matrix (and others by tuning its inputs: the parameterization $\phi$ and the integer $\nu$):
{\small 
\begin{verbatim}
	>A=QQ[s0,s1,s2];
	>f0=s0^2+s1^2+s2^2; f1=2*s0*s2; f2=2*s0*s1; f3=s0^2-s1^2-s2^2; 
	>F=matrix{{f0,f1,f2,f3}};
	>Z1=kernel koszul(1,F);
	>R=A[X0,X1,X2,X3];
	>d=(degree f0)_0; nu=2*(d-1)-1;
	>Z1nu=super basis(nu+d,Z1);
	>Xnu=matrix{{X0,X1,X2,X3}}*substitute(Z1nu,R);
	>Bnu=substitute(basis(nu,A),R); 
	>(m,M)=coefficients(Xnu,Variables=>{s0_R,s1_R,s2_R},Monomials=>Bnu);
	>M -- this is the matrix representation in degree nu
\end{verbatim}
}

A primary decomposition of the $3\times 3$ minors of $M(\phi)_1$, i.e.~$\Fitt^0_1(\phi)$, returns 
$$(X_0^2-X_1^2-X_2^2-X_3^2)\cap(X_2,X_1,X_0^2+2X_0X_3+X_3^2)$$
which corresponds to the implicit equation of the sphere plus one embedded double point $(1:0:0:-1)$. Now, a primary decomposition of the $2\times 2$ minors of $M(\phi)_1$, i.e.~$\Fitt^1_1(\phi)$, returns 
$$(X_2,X_1,X_0+X_3)\cap (X_3,X_2^2,X_1X_2,X_0X_2,X_1^2,X_0X_1,X_0^2)$$
which corresponds to the same embedded point $(1:0:0:-1)$, now with multiplicity one, plus an additional component supported at the origin. Finally, the ideal of $1$-minors of $M(\phi)_1$, i.e.~$\Fitt^2_1(\phi)$, is supported at the origin (i.e.~is empty as a subscheme of $\PP^2$).

The point $(1:0:0:-1)$ is actually a singular point of the parameterization $\phi$ (but not of the sphere itself). Indeed, the line $L=(0:s_1:s_2)$ is a $\PP^1$ that is mapped to $(s_1^2+s_2^2: 0:0:-(s_1^2+s_2^2))$. In particular, the base points of $\phi$, $(0:1:i)$ and $(0:1:-i)$, are lying on this line, and the rest of the points are mapped to the point $(1:0:0:-1)$. Outside $L$ at the source and $P$ at the target, $\phi$ is an isomorphism.
\end{exmp}

\subsection{Regularity of the symmetric algebra}
Hereafter, we give an upper bound for the regularity of the symmetric algebra $\Sc_I$ in our setting. We will use in the course of this paper. We begin with some classical notation. 

Given a finitely generated graded $S$-module $N$, we will denote by $HF_N$ its Hilbert function and by $HP_N$ its Hilbert polynomial. Recall that for every $\mu$, 
\[
HF_N(\mu ):=\dim_k (N_\mu ),
\]
and
\begin{equation}\label{eq:HPHF}
	HP_N (\mu ):=HF_N(\mu )-\sum_i (-1)^i HF_{H_{\mm}^i(N)}(\mu)	
\end{equation}
where $\mm$ denotes the ideal $\mm=(s_0,s_1,s_2)\subset S$.
It is well-known that $HF_{H_{\mm}^i(N)}(\mu)$ is finite for any $i$ and $\mu$, and that $HP_N$ is a polynomial function.

We also recall the definition of three classical invariants attached to a graded module $N$. We will denote 
$$a^i(N):=\sup\{\mu\ |\ H^i_\mm (N)_\mu\neq 0\}, \ \  a^*(N):=\max_i\{a^i(N)\}, \ \ $$ 
and $$\reg(N):=\max_i\{a^i(N)+i\}$$
for the $a^i$-invariant, the $a^*$-invariant and the Castelnuovo-Mumford regularity of $N$, respectively.
Finally, the notation $\mu(J)$ for a given ideal $J$ stands for its minimal number of generators.

The following lemma will be useful in some places (see \cite[Corollary A12]{SZ} for a self contained elementary proof) :

\begin{lem}\label{regdim0}
Let $J$ be a graded ideal in  $S$, generated in degree $d$ such 
that $\dim(S/J)\leq 1$ and $\mu(J_\pp)\leq 3$ for every prime ideal $\mm\supsetneq \pp \supset J$. Then, 
$$\reg (S/J^{sat})\leq 2d-3$$ 
and
$$
\reg (S/J)\leq 3d-3-\indeg (J^{sat})
$$
unless $\mu (J)=2$, in which case $J=J^{sat}$ and $\reg (S/J)=2d-2$.
\end{lem}

\begin{proof}
We may assume that $k$ is infinite. 
Let $f$ and $g$ be two general forms of degree $d$ in $J$. 
They form a complete intersection. Hence, by liaison (see e.g.~\cite[4.1 (a)]{CU02}) 
$$\reg(S/J^{sat})=\reg(S/(f,g))-\indeg(((f,g):J^{sat})/(f,g)).$$
Hence $\reg (S/J^{sat})\leq 2d-3$, unless $(f,g)\supseteq J^{sat}$ in which case $J=(f,g)$.

Let $J'\subseteq J$ be an ideal generated by $3$ general forms of degree $d$. 
As $\mu(J_\pp)\leq 3$ for every prime ideal $\mm\supsetneq \pp $,
$J'$ and $J$ have the same saturation. Therefore, it suffices to show that $H^0_\mm (S/J')_{\mu +d}=0$. Observe that by Koszul duality (see for instance \cite[Lemma 5.8]{Ch04}) 
$H^0_\mm (S/J')_{\mu +d}=0$ is equivalent to $H^0_\mm (S/J')_{2d-3-\mu}=0$. But since 
$2d-3-\mu <d$ we indeed have 
$$H^0_\mm (S/J')_{2d-3-\mu}=(J')^{sat}_{2d-3-\mu}=J^{sat}_{2d-3-\mu}=0$$ 
since  $2d-\mu -3<\indeg ( J^{sat})$.
\end{proof}

\begin{prop}\label{prop:regSym} Assume that $\dim(S/I)\leq 1$ and $\mu(I_\pp)\leq 3$ for every prime ideal $\mm\supsetneq \pp \supset I$, then $\Zc_\bullet$ is acyclic, $a^*(\Sc_I)+1\leq\nu_0:=2d-2-\indeg(I^{sat})$, and
$$
\reg(\Sc_I)\leq \nu_0,
$$ 
unless $I$ is a complete intersection of two forms of degree $d$\footnote{Notice that in this case the  closed image of $\phi$ cannot be a surface.}, in which case
$a^*(\Sc_I)= d-3$ and  $\reg(\Sc_I)=d-1=\nu_0+1$.
\end{prop}
\begin{proof} The acyclicity of $\Zc_\bullet$ is proved in \cite[Lemma 4.2]{BC05}. Furthermore, $H^i_\mm(\Sc_I)=0$ for $i\notin \{0,1,2,3\}$, so we only need to examine $H^i_\mm(\Sc_I)$ for $i\in \{0,1,2,3\}$. The case $i=0$ is treated in \cite[Theorem 4.1]{BC05}. The proof  uses the spectral sequence
\begin{equation}\label{eq:spectseq}
	H^j_\mm (Z_{i+j})_{\mu +(i+j)d}\otimes_k R\Rightarrow H^i_\mm ( \Sc_I)_\mu	
\end{equation}
from which we will also deduce the vanishing for $i>0$.

\medskip


For $i=1$, let $\mu \geq \nu_0$. We will show that $H^2_\mm (Z_1)_{\mu +d}=H^3_\mm (Z_2)_{\mu +2d}=0$ as by \eqref{eq:spectseq} it implies that $H^1_\mm(\Sc_I)_\mu=0$  as claimed (recall that $H^1_\mm(Z_0)=H^1_\mm(S)=0$).

By \cite[Lemma 4.1]{BC05}, $H^3_\mm (Z_2) \simeq (Z_1)^*[3-4d]$ and we deduce that
$$H^3_\mm (Z_2)_{\mu +2d}\simeq \left((Z_1)^*[3-4d]\right)_{\mu+2d}\simeq (Z_1)^*_{\mu-2d+3}.$$
But $2d-\mu -3< \indeg ( I^{sat})\leq d$ as $\mu\geq \nu_0$. Hence $(Z_1)_{2d-\mu -3}=0$ since $Z_1\subset S(-d)^4$, proving that $H^3_\mm (Z_2)_{\mu +2d}=0$. 

The exact sequence $0\to Z_1\to S(-d)^4\to S\to S/I\to 0$ provides a graded isomorphism $H^2_\mm(Z_1) \simeq  H^0_\mm (S/I)$.
By Lemma \ref{regdim0}, $\reg (S/I)\leq 3d-3-\indeg (I^{sat})=d+\nu_0 -1$, hence  $H^2_\mm (Z_1)_{\mu +d}=0$.
\medskip

For $i=2$, consider the canonical exact sequence defining $K$
$$ 0 \rightarrow K \rightarrow \Sc_I \rightarrow \Rc_I \rightarrow 0.$$
As $K$ is the direct sum of the $S$-modules $\ker (\Sym^t_S(I) \to I^t)$ that are supported on $V(I)$, we have $H^i_\mm (K)=0$ for $i>\dim S/I$.
It follows that $H^2_\mm(\Sc_I) \simeq H^2_\mm(\Rc_I)$. Now, $a^2(\Rc_I)=\max_{t\geq 1}\{ a^2(I^t)-td\}=\max_{t\geq 1}\{ a^1(S/I^t)-td\}$. 
The surjective map $(S/I)(-td)^{\binom{t+3}{3}}\to I^t/I^{t+1}$ provides the inequality $a^1(I^t /I^{t+1}) \leq a^1(S/I)+td$, and it 
remains to estimate $a^1(S/I)$. 
By Lemma \ref{regdim0}, 
$$a^2(\Rc_I)\leq a^1(S/I) - d\leq  d-4,$$
unless $\mu (I)=2$.  As $\indeg(I^{sat})\leq d$ it follows that 
$a^2(\Rc_I)\leq \nu_0-2$, unless  $I=(f,g)$, in which case $a^2(\Rc_I) =\nu_0-1$. 

\medskip

For $i=3$, $H^3_\mm(Z_0)_\mu =H^3_\mm(S)_\mu=0$ for all $\mu>-3$, hence $a^3(\Sc_I)\leq -3$. Finally, $a^3(\Sc_I)=\oplus_{t\geq 0} H^3_\mm(S)(td)$, hence $a^3(\Sc_I)=-3$.
\end{proof}

\begin{rem}
If $V(I)=\emptyset$ then $a^2 (\Sc_I)=-\infty$ and $a^1 (\Sc_I)=a^1(I)-d=\reg (S/I)-d$.
Also, when $\dim(S/I)=1$ and $V(I)$ is locally a complete intersection, one has
$a^2 (\Sc_I)=a^2(I)-d$ as above and $a^1 (\Sc_I)=\max\{ a^1(I)-d,a^1(I^2)-2d\}$.
This shows in particular that $a^1 (\Sc_I)$ and $a^2 (\Sc_I)$ can be (very quickly) computed 
from $I$, using a dedicated software.
\end{rem}

\section{Fibers of the canonical projection of the graph of $\phi$ onto $\Sc$}\label{sec:main}

Let $\Gamma \subset \PP^2_{\PP^3}$ be the graph of $\phi$ and $R_{+}=(X_0,\ldots,X_3)$ the graded maximal ideal of $R$.  
We have the following diagram
\begin{equation}\label{eq:diagram}
 \xymatrix@1{\Gamma\ \ar[d]_{\pi_1}\ar[rd]^{\pi_2}\ar@{^(->}[r] & \PP^2_{\PP^3_k} \\  \PP^2_k & \PP^3_k}	
\end{equation}
where all the maps are canonical. For any $\pp\in \Proj(R)$, we will denote by $\kappa(\pp)$  its residue field $R_\pp/\pp R_\pp$. 
The fiber of $\pi_2$ at $\pp\in \Proj (R)$ is the subscheme
\[
\pi_2^{-1}(\pp):=\Proj(\Sc_I\otimes_R \kappa(\pp)) \subset \PP^2_{\kappa(\pp)}.
\]
For simplicity, we set $N_\nu^\pp:=\dim_{\kappa (\pp )}(\Sc_I\otimes_{R} \kappa(\pp))_{\nu}$ for the Hilbert function of the fiber $\pi_2^{-1}(\pp)$ in degree $\nu$. It turns out that this quantity governs the support of the Fitting ideals $\Fitt^i_\nu(\phi)$ we are interested in.

\begin{prop}\label{propVanishingFitt}
 For any $\pp\in \Proj (R)$  
\[
\Fitt^i_\nu(\phi)\otimes_R \kappa(\pp)= 
\begin{cases}
	0 & \textrm{ if } \ 0\leq  i<N_\nu^\pp \\
	\kappa(\pp) & \textrm{ if } \  N_\nu^\pp \leq i
\end{cases}
\]
\end{prop}
\begin{proof} First, as a consequence of the stability under base change of the Fitting ideals, one has 
\[
\Fitt^i((\Sc_I)_\nu)\otimes_{R} \kappa(\pp)=\Fitt^i((\Sc_I)_\nu\otimes_{R} \kappa(\pp)).
\]
Now, for any $\nu\geq 0$ the $\kappa(\pp)$-vector space 
\[
(\Sc_I)_\nu\otimes_{R} \kappa(\pp)=(\Sc_I\otimes_{R} \kappa(\pp))_\nu
\]
 is of dimension $N_\nu^\pp$. The result then follows from the fact that the $i$-th Fitting ideals of a $\kappa(\pp)$-vector space of dimension $N$ is
 $0$ for $i<N$ and $\kappa(\pp)$ for $i\geq N$.
\end{proof}

This leads us to focus on the behavior of the Hilbert functions  $N_\nu^\pp$. By \eqref{eq:HPHF}, it suffices to control  the local cohomology of the fiber $\Sc_I\otimes \kappa(\pp)$ with respect to the ideal $\mm$. Part of this control is already known,

\begin{lem}[{\cite[Prop.\ 6.3]{Cha12}}]\label{lemFromDimToN}
Let $\pp\in \Proj (R)$ and $v_\pp:=\dim (\Sc_I\otimes_R\kappa(\pp))$. Then,
$$
a^{v_\pp}(\Sc_I\otimes_R\kappa(\pp))=a^{v_\pp} (\Sc_I\otimes_R R_\pp)\leq a^{v_\pp}(\Sc_I) \\
$$
and 	
$$
a^{v_\pp -1}(\Sc_I\otimes_R\kappa(\pp))\leq \max\{ a^{v_\pp -1}(\Sc_I\otimes_R R_\pp), a^{v_\pp} (\Sc_I\otimes_R R_\pp)\} 
\leq \max\{ a^{v_\pp -1}(\Sc_I), a^{v_\pp} (\Sc_I)\} .
$$
\end{lem}


We are now ready to describe the support of the ideals $\Fitt^i_\nu(\phi)$ with $i\geq 0$ and $\nu\geq \nu_0$. We begin with the points on $\Scc$ whose fiber is 0-dimensional.

\begin{thm}[0-dimensional fibers] Let $\pp \in \Proj (R)$.
	If $\dim \pi_2^{-1}(\pp) = 0$ then for all $\nu\geq \nu_0$
	\begin{equation}\label{eqFittDeg}
	\pp \in V(\Fitt^i_\nu(\phi)) \Leftrightarrow \deg(\pi_2^{-1}(\pp))\geq i+1.
	\end{equation}
\end{thm}

\begin{proof} By Proposition \ref{prop:regSym} and  Lemma \ref{lemFromDimToN}, we have $H^i_\mm(\Sc_I\otimes_R\kappa(\pp))_\nu=0$ for all $\nu\geq \nu_0$ and $i\geq 0$. Then, the conclusion follows from Proposition \ref{propVanishingFitt} and the equality \eqref{eq:HPHF} since the Hilbert polynomial of the fiber $\pi_2^{-1}(\pp)$ is a constant polynomial which is equal to the degree of this fiber. 
\end{proof}

We now turn to points on $\Scc$ whose fiber is 1-dimensional.

\begin{lem}\label{1fiber}
Assume that the $f_i$'s  are linearly independent, the fiber over a closed point $\pp$ of coordinates $(p_0:p_1:p_2:p_3)$ is of dimension 1, and its
unmixed component is defined by $h_\pp\in S$. 
Let $\ell_\pp$ be a linear form with 
$\ell_\pp (p_0,\ldots ,p_3 )=1$ and set  $\ell_i (x):=x_i-p_i \ell_\pp (x)$.  
Then, 
$h_\pp =\gcd (\ell_0(f),\ldots ,\ell_3(f))$ 
and 
$$
I=(\ell_\pp (f))+h_\pp (g_0,\ldots ,g_3)
$$ 
with $\ell_i (f)=h_\pp g_i$ and   $\ell_\pp (g_0,\ldots ,g_3 )=0$.
In particular 
$$
(\ell_\pp (f))+h_\pp (g_0,\ldots ,g_3 )^{sat}\subseteq I^{sat} \subseteq (\ell_\pp (f))+(h_\pp ).
$$
\end{lem}

\begin{proof}
A syzygy $L=\sum_{i} a_{i}x_i =\sum_{i} a_{i}(p_i \ell_\pp (x)+\ell_i (x))$
provides an equation for the fiber : $\overline{L}=\sum_{i} a_i p_{i}$. 
Recall that $h_\pp =\gcd (\overline{L_1},\ldots ,\overline{L_t})$ where
the $L_j$ are generators of the syzygies of $I$.

The particular syzygy $L_{(i)}:=\ell_i (f)\ell_\pp (x)-\ell_\pp (f)\ell_i (x)$ 
satisfies  $\overline{L_{(i)}}=\ell_i (f)$. It follows that $h_\pp$ divides  $h:=\gcd (\ell_0(f),\ldots ,\ell_3(f))$.
Set $\ell_i (f)=hg_i$ for some $g_i\in S$. 
The $\ell_i$'s span the linear forms vanishing at $p$, and are related by 
the equation $\ell_\pp (\ell_0,\ldots ,\ell_3 )=0$. It follows that
$$
I=(\ell_\pp (f),\ell_0(f),\ldots ,\ell_3 (f))=(\ell_\pp (f))+h(g_0,\ldots ,g_3),
$$
and $\ell_\pp (g_0,\ldots ,g_3)=0$. 
Now if one has a relation $a\ell_\pp (f)+\sum_i b_i hg_i=0$,
then $h$ divides $a$, as  the $f_i$'s have no 
common factor. This in turn shows that $h$ divides $h_\pp$ and completes the proof.
\end{proof}

Notice that  if $p_i\not= 0$, $g_i$ is a linear combination of the 
$g_j$'s for $j\not= i$ as $\ell_\pp (g_0,\ldots ,g_3)=0$. For instance if $p_0\not= 0$,
$I=(\ell_\pp (f))+h(g_1,\ldots ,g_3)$.\medskip

\begin{rem}\label{rem:NoBPfinite}
	The above lemma shows that fibers of dimension 1 can only occur 
	when $V(I)\not= \emptyset$ as $V(I)\supseteq V((\ell_\pp (f),h_\pp ))$. It also 
	shows that 
	$$
	d\deg (h_\pp )\leq \deg (I)
	$$
	if there is a fiber of dimension 1 of unmixed part given by $h_\pp$.
	Furthermore, any element $F$ in $I^{sat}$ of degree $<d$ has to be a multiple
	of $h_\pp$ for any fiber of dimension 1. As $h_\pp$ and $h_{\pp'}$ cannot have a 
	common factor for $\pp \not= \pp'$, the product of these forms divide $F$.
	This gives a simple method to compute 1 dimensional fibers if such an $F$ exists.	
\end{rem}

%


\begin{thm}[1-dimensional fibers]\label{thm:1dimfibers} Let $\pp \in \Proj (R)$ be a $k$-rational closed point.
 If $\dim \pi_2^{-1}(\pp) = 1$ then, setting $\delta:=\deg (\pi_2^{-1}(\pp)))$,
 		\begin{equation*}
			a^0 ( \Sc_I\otimes \kappa(\pp))\leq 	 \nu_0 -\delta -1
		\end{equation*}
and
\begin{equation*}
			\reg ( \Sc_I\otimes \kappa(\pp))\leq 	\max\{ \nu_0 -\delta -1,\reg (I)-d\}\leq \nu_0 -1.
		\end{equation*}
		As a consequence, $a^* ( \Sc_I\otimes \kappa(\pp))\leq 	 \nu_0 -2$.
\end{thm}
\begin{proof} 
By Lemma \ref{1fiber}, $I=(f)+h(g_1,g_2,g_3)$ 
 with $h$ the equation of the unmixed part of the fiber. Let $I':=(g_1,g_2,g_3)$. 
	
	We may assume for simplicity that $\pp=(1:0:0:0)$ (i.e. $f=f_0$ and $f_i=hg_i$ for $i=1,2,3$), and we then have an exact sequence
	$$ 
	 0 \rightarrow Z_1'(d) \rightarrow Z_1(d) \rightarrow S \rightarrow \Sc_I\otimes \kappa(\pp) \rightarrow 0,$$
	 with $Z_1'$ the syzygy module of $(f_1,f_2,f_3)$, the first map sending $(b_1,b_2,b_3)$ to
	 $(0,b_1,b_2,b_3)$ and the second  $(a_0,a_1,a_2,a_3)$ to $a_0$.

	It follows that $\reg ( \Sc_I\otimes \kappa(\pp))\leq \max\{ \reg (S),\reg (Z_1)-d-1,\reg (Z'_1)-d-2\}$
	and $a^0 ( \Sc_I\otimes \kappa(\pp))\leq a^2(Z'_1)-d\leq \reg (Z'_1)-d-2$ as $H^0_\mm (S)=H^1_\mm (Z_1)=0$.
	
By Lemma \ref{1fiber}, $h=\gcd (f_1,f_2,f_3)$ is of degree $\delta$. Set $f_i=hf_i'$ for all $i=1,2,3$. 
The syzygies of  $(f_1,f_2,f_3)$ and  $I':=(f_1',f_2',f_3')$ are equal, up to a degree shift by $\delta$,
and the exact sequence $0\to Z'_1 (\delta )\to S(-d+\delta )^3 \to I'\to 0$ shows that :
$$
\reg (Z'_1)=\reg (I')+1+\delta .
$$
Notice that $hI'\subset I$, hence $h(I')^{sat}\subseteq I^{sat}$ and
$\indeg (I^{sat})\leq \delta +\indeg ((I')^{sat})$. By Lemma \ref{regdim0} we obtain
$$
\reg (I')\leq 3(d-\delta)-2-\indeg ((I')^{sat})\leq 3d -2\delta -\indeg (I^{sat}) -2
$$
hence
$$
a^0 ( \Sc_I\otimes \kappa(\pp))\leq 2d-\delta -3 -\indeg (I^{sat})=\nu_0 -\delta -1.
$$ 
Finally, $\reg (Z_1)-d-1=\reg (I)-d\leq  \nu_0 -1$ by Lemma \ref{regdim0}.
\end{proof}


	To conclude the above case discussion, let us state a simple corollary :
	
\begin{cor}\label{cor:key}
Let $\pp \in \Proj (R)$.\\

(i) If  $\pi_2^{-1}(\pp)=\emptyset$ then $a^* ( \Sc_I\otimes \kappa(\pp))=\reg ( \Sc_I\otimes \kappa(\pp))\leq \nu_0 -1$.\\

(ii)  If  $\dim (\pi_2^{-1}(\pp))=0$ then   $a^* ( \Sc_I\otimes \kappa(\pp))\leq \nu_0 -1$ and $\reg ( \Sc_I\otimes \kappa(\pp))\leq  \nu_0$.\\

(iii)  If  $\dim (\pi_2^{-1}(\pp))=1$ then   $a^* ( \Sc_I\otimes \kappa(\pp))\leq \nu_0 -2$ and $\reg ( \Sc_I\otimes \kappa(\pp))\leq  \nu_0-1$.

\end{cor}

Let us now remark that there is no fiber of dimension two (i.e. equal to $ \PP^{2}$) over a point  $\pp \in \Proj (R)$. More generally :

\begin{lem}\label{FibreP2} 
Let $\phi: \PP^{n-1}_k \dto \PP^{m}_k$ be a rational map and $\Gamma$ be the closure of the graph of $\phi$.
The following are equivalent :
(i) $\pi : \Gamma \rightarrow   \PP^m_k$ has a fiber $F_x$ equal to $\PP^{n-1}_x$,
(ii) $\phi$ is the restriction to the complement of a hypersurface of the constant map
sending any point to the zero dimensional and  $k$-rational point $x$. 
\end{lem}

\begin{proof}
For (i)$\Rightarrow$(ii), notice that $ \PP^{n-1}_k\times \{ x\}$ is a subscheme of the absolutely irreducible 
scheme $\Gamma$ that has same dimension $n-1$. Hence $F_x =\Gamma$ showing (ii). 
\end{proof}

\begin{thm}[2-dimensional fibers]\label{fibreP2} 
For all $\nu\geq \nu_0$,
$$\left\{\pp\in \Proj (R) : HP_{\pi_2^{-1}(\pp)}(\nu) = HP_{\PP^2}(\nu) =\binom{\nu+2}{2} \right\}=\emptyset$$
\end{thm}

\begin{proof}
It follows from Proposition \ref{prop:regSym} that the defining ideal of the fiber is generated in degree at most
$\nu_0$. As this ideal is non trivial for any point in $\Proj (R)$ by Lemma \ref{FibreP2}, we deduce that $HP_\nu^\pp < \binom{\nu+2}{2}$
for any prime in $\Proj (R)$ if $\nu \geq \nu_0$.
\end{proof}

It is worth mentioning two facts that are direct consequences of the above results. First, the embedded components of $\Fitt^0_\nu(\phi)$ are exactly supported on $V(\Fitt^1_\nu(\phi))\subset V(\Fitt^0_\nu(\phi))=\Scc$. Second, the set-theoretical support of each $\Fitt^i_\nu(\phi)$, $i$ fixed, stabilizes for $\nu \gg 0$; its set-theoretical support then correspond to those points on $\Sc$ whose fiber is either a curve and or a finite scheme of degree greater or equal to $i+1$.

\section{Computational aspects}\label{sec:comp}

In this section, we detail the consequences of the previous results for giving a computational description of 
the singular locus of the parameterization $\phi$ of the surface $\Sc$.

\subsection{Description of the fiber of a given point on $\Scc$}\label{sec:algo}

Given a point $\pp$ on $\Scc$, we summarize the results we obtained with Figure \ref{figure}.
The blue curve corresponds to a finite fiber, i.e.~a point $\pp$ such that  $\dim \pi_2^{-1}(\pp)=0$, the red curve corresponds a one dimensional fiber, i.e.~ a point $\pp$ such that  $\dim \pi_2^{-1}(\pp)=1$, and the black curve represents the Hilbert polynomial of $S$, the coordinate ring of $\PP^{2}_{k}$.

\begin{figure}[!ht]
	\includegraphics{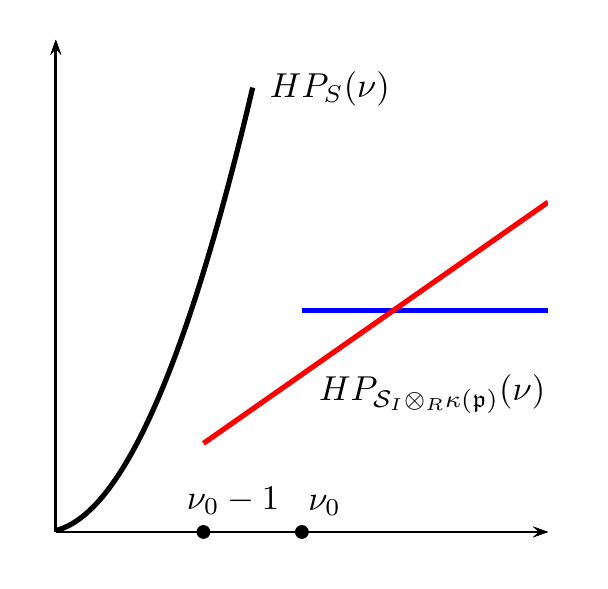}
	\caption{Graph of $N_\nu^\pp$ as a function of $\nu$ for a given $\pp \in \Proj(R)$.}
	\label{figure}
\end{figure}

The fact that $\pp$ belongs or not to the support of $\Fitt_\nu^i(\phi)$ can be checked by the computation of the rank of a matrix representation $M_\nu(\phi)$ evaluated at the point $\pp$. More precisely, we get the following properties:
\begin{itemize}
	\item If $\dim \pi_2^{-1}(\pp)\leq 0$ then for all $\nu\geq \nu_0$ 
	$$\corank(M_\nu(\phi)(\pp))=\deg(\pi_2^{-1}(\pp)).$$
	\item If $\dim \pi_2^{-1}(\pp) = 1$ then for all $\nu\geq \nu_0-1$
	$$\corank(M_\nu(\phi)(\pp))=\deg(\pi_2^{-1}(\pp))\nu+c$$
	where $c\in \ZZ$ is such that $HP_{\pi_2^{-1}(\pp)}(\nu)=\deg(\pi_2^{-1}(\pp))\nu+c$.
\end{itemize}

From here, we get an algorithm to determine the Hilbert polynomial of a fiber $\pi_2^{-1}(\pp)$ of a given point $\pp$ by comparing the rank variation of the matrix $M_\nu(\phi)(\pp)$ for two consecutive integers $\nu$ and $\nu+1$, assuming $\nu \geq \nu_0$. Before giving a more precise description in the form of an algorithm, we refine the distinction between finite and non-finite fibers.

\begin{prop} Let $\pp$ be a point on $\Scc$ and choose an integer $\nu$ such that $$\nu_0\leq \nu\leq 2d-2.$$ 
If $\corank(M_{\nu}(\phi)(\pp))\leq \nu$ then $\pi_2^{-1}(\pp)$ is necessarily a finite fiber.
\end{prop}
\begin{proof} We know that the Hilbert function of the fiber $\pi_2^{-1}(\pp)$ coincides  with its Hilbert polynomial for all $\nu\geq \nu_{0}$. Therefore, to prove our claim it is enough to show that if the fiber $\pi_2^{-1}(\pp)$ of $\pp$ is of dimension 1, then 
$$\corank(M_{\nu}(\phi)(\pp))=HP_{\pi_2^{-1}(\pp)}(\nu)\geq \nu+1.$$ 
So, assume that $\pi_2^{-1}(\pp)$ is 1-dimensional and set $\delta:=\deg(\pi_2^{-1}(\pp))$ so that 
	$$  HP_{\pi_2^{-1}(\pp)}(\nu)=\delta\nu+1-\binom{\delta-1}{2}+N$$
where $\binom{\delta-1}{2}$ is the arithmetic genus of the unmixed dimension one part of the fiber $\pi_2^{-1}(\pp)$ and $N\geq 0$ is the degree of its remaining finite part. Observe that $\delta\leq d$. 
Now, a straightforward computation yields
$$HP_{\pi_2^{-1}(\pp)}(\nu)=\frac{\delta(2\nu+3-\delta)}{2}+N.$$
From here, we see that $HP_{\pi_2^{-1}(\pp)}(\nu)$ is an increasing function of $\delta$ on the interval $[-\infty;\nu+\frac{3}{2}]$
which takes value ${\nu+1+N}$ for $\delta=1$. Observe moreover that $\nu+\frac{3}{2}\geq \delta$ as soon as  $\nu \geq d-1$. And if $\nu=d-2$, then $HP_{\pi_2^{-1}(\pp)}(d-2)$ takes the same value for $\delta=d-1$ and $\delta=d$.
\end{proof}

An interesting consequence of the above proposition is that for all integer $0\leq n \leq 2d-2$
\begin{equation*}
	\corank(M_{2d-2}(\phi)(\pp))=n \Leftrightarrow \pi_2^{-1}(\pp) \textrm{ is finite of degree } n.	
\end{equation*}
Therefore, as soon as $d\geq 2$ we have
$$\corank(M_{2d-2}(\phi)(\pp))=1 \Leftrightarrow \pp \textrm{ has a unique preimage}.$$
Observe in addition that once we know that a point $\pp$ has a unique preimage, this preimage can be computed from the kernel of the transpose of $M_{2d-2}(\phi)(\pp)$. Indeed, this kernel is generated by a single vector. But it is clear by the definition of $M_{2d-2}(\phi)$ that the evaluation of the basis of $S_\nu$ chosen to build $M_{2d-2}(\phi)$ evaluated at the point $(s_0:s_1:s_2)$ is also a nonzero vector in the kernel of the transpose of $M_{2d-2}(\phi)(\pp)$. Hence, from this property it is  straightforward to compute the pre-image of $\pp$ if it is unique.

%
%
%
 \medskip
 
 \noindent \hrulefill
 
  \textbf{Algorithm} : Determination of the characters of a fiber $\pi_2^{-1}(\pp)$

 \noindent \hrulefill 
 
 \medskip
 
  \textbf{Input:} a parameterization $\phi$ of a surface $\Scc$ satisfying \eqref{H} and a point $\pp \in \PP^3$.
  
  \textbf{Output:} the hilbert polynomial of the fiber $\pi_2^{-1}(\pp)$.
 
\medskip

\begin{enumerate}
	\item[1.] Pick an integer $\nu\geq \nu_0$ (e.g.~$\nu=2d-2$ which is always valid).
	\item[2.]  Compute a matrix representation $M_{\nu}(\phi)$.
	\item[3.] Compute $r_{\nu}:=\corank(M_{\nu}(\phi))(\pp)$.
	\item[4.] If $r_{\nu}=0$ then \emph{$\pp$ does not belong to $\Scc$} (and stop).
	\item[5.] If $0<r_{\nu}\leq \nu$ then $\pi_2^{-1}(\pp)$ is a finite set of $r_{\nu}$ points, counted with multiplicity.
	\item[6.] If $r_{\nu}>\nu$ then compute $M_{\nu+1}(\phi)$.
	\item[7.] Compute $r_{\nu+1}:=\corank(M_{\nu+1}(\phi))(\pp)$.
	\item[8.] If $r_{\nu}=r_{\nu+1}$ then $\pi_2^{-1}(\pp)$ is a finite set of $r_{\nu}$ points, counted with multiplicity.
	\item[9.] If $r_{\nu}<r_{\nu+1}$ then $\dim(\pi_2^{-1}(\pp))=1$ and 
	\begin{equation*}
	HP_{\pi_2^{-1}(\pp)}(\nu)=(r_{\nu+1}-r_{\nu})\nu+r_{\nu}-(r_{\nu+1}-r_{\nu})\nu.		
	\end{equation*}
	In other words, $\pi_2^{-1}(\pp)$ is made of a curve of degree $r_{\nu+1}-r_{\nu}$ and of a finite set of  
	\begin{equation*}
	\binom{r_{\nu+1}-r_{\nu}-1}{2}+r_{\nu}-(r_{\nu+1}-r_{\nu})\nu-1.	
	\end{equation*}
\end{enumerate}

 \noindent \hrulefill
 
 \medskip

\subsection{Pull-back to the parameter space}

So far, we have seen that the fibers of $\pi_2$ can be described from the Fitting ideals $\Fitt^i_\nu(\phi)$ which are homogeneous ideals in the ring $R$ of implicit variables. However, for applications in geometric modeling, it is also  interesting to get such a description of the singular locus of $\phi$ in the source space $\PP^2$ in place of the target space $\PP^3$. From a computational point of view, this also allows to reduce the number of ambient variables by one.
 For that purpose, we just have to pull-back our Fitting ideals through $\phi$. 
 For all integers $i\geq 0$ and $\nu\geq \nu_0$ we define the ideals of $S$
	$$ \phi^{-1}(\Fitt_\nu^i):=\phi^\sharp(\Fitt_\nu^i).S$$	

It is clear that for all $i$ we have $\phi^{-1}(\Fitt_\nu^i) \subset I$, that is to say that the base points of $\phi$ are always contained in the support of $\phi^{-1}(\Fitt_\nu^i)$ for all $i\geq 0$ and all $\nu\geq \nu_0$. By Lemma \ref{fibreP2}, this support is exactly the base points of $\phi$ if $i=\binom{\nu+2}{2}-1$. In other words we have $$\phi^{-1}\left(\Fitt_\nu^{\binom{\nu+2}{2}-1} \right)=I, \hspace{.2cm} \forall \, \nu \geq \nu_0.$$

\medskip

For all $\nu\geq \nu_0$, we clearly have $\phi^{-1}(\Fitt_\nu^0)=(0) \subset S$. Notice that this behavior is similar to substituting the parameterization $\phi$ into an implicit equation of $\Scc$ which returns $0$. However, $\phi^{-1}(\Fitt_\nu^1)$ is nonzero since we have assumed $\phi$ birational. Actually, as a consequence of the results of Section \ref{sec:main}, the sequence of ideals 
$$\phi^{-1}(\Fitt_\nu^1) \subset \phi^{-1}(\Fitt_\nu^2) \subset \cdots \subset  \phi^{-1}\left(\Fitt_\nu^{\binom{\nu+2}{2}-1} \right)=I$$
yields a filtration of $\PP^2$ which is in correspondence with the singularities of $\phi$. Here is an illustrative example where we focus on the singularities of the parameterization $\phi$ in a situation where $\Scc$ is smooth and very simple.

\begin{exmp} Take four general linear forms $l_0,l_1,l_2,l_3$ in $S_1$ and consider the parameterization
$$\phi: \PP^2_\CC \rightarrow \PP^3_\CC: (s_0:s_1:s_2) \rightarrow (l_1l_2:l_0l_3:l_0l_1:l_0l_1).$$
The closed image of $\phi$ is the plane of equation $\Scc: X_2-X_3=0$. There are three base points that are all local complete intersections: 
$$P_{01}=\{l_0=0\}\cap \{l_1=0\}, \ P_{02}=\{l_0=0\}\cap \{l_2=0\}, \ P_{13}=\{l_1=0\}\cap \{l_3=0\}.$$ 
The fiber of each base point by $\pi_1$ (see the diagram \eqref{eq:diagram})  gives a line on $\Scc$, the projection of the associated exceptional divisor, say $D_{01}, D_{02}$ and $D_{13}$. The equations of the graph $\Gamma$ are $l_0X_0-l_2X_2,l_1X_1-l_3X_2,(l_3-l_2)X_2, X_2-X_3$ so that we have
$$D_{01}: \{X_2=0\}\cap \Scc, \ D_{02}: \{l_1(P_{02})X_1-l_3(P_{02})X_2=0\}\cap \Scc,$$ 
$$D_{13}: \{l_0(P_{13})X_0-l_2(P_{13})X_2=0\}\cap \Scc.$$
Now, let $P$ be a point on $\Scc$. If $P \notin D_{01}\cup D_{02} \cup D_{13}$ then $\pi_1(\pi_2^{-1}(P))$ does not contain any base point and is necessarily zero-dimensional. Moreover, 
$\pi_1 \pi_2^{-1}(D_{01}\cap D_{13})=\{ l_1=0\}$, so that we have the one-dimensional fiber over $D_{01}\cap D_{13}=(0:1:0:0)$ that goes through only two of the three base points. A similar property appears for $\pi_1 \pi_2^{-1}(D_{01}\cap D_{02})=\{ l_0=0\}$.

Notice that the three lines $D_{01},D_{02}, D_{13}$ are not in the image of $\phi$ except for the two points: $(0:1:0:0)$ and $(1:0:0:0)$. Moreover, the fibers over these two points are lines and they are the only points with a one-dimensional fiber. 
\end{exmp}

\section{Link with multiple-point schemes of finite maps}\label{sec:link}

Given a finite map of schemes $\varphi:X\rightarrow Y$ of codimension one, its source and target double-point cycles have been extensively studied in the field of intersection theory (see e.g.~\cite{KLU96,KLU92,Pie78,Tei77}). In \cite{MP89}, Fitting ideals are used to give a scheme structure to the multiple-point loci. 

Under the hypothesis that $\phi:\PP^2\rightarrow \PP^3$ is a finite map, our results are partially contained in the above literature. Recall that if $\phi$ has no base point then $\phi$ is finite (see Remark \ref{rem:NoBPfinite}).
The contribution of our paper is hence mainly an extension to the case where $\phi$ is not finite. It is known that it is always possible to remove the base points by blowing-up, and this is exactly the role of the symmetric algebra in our approach of matrix representations, but still some one-dimensional fibers might remain. 

In loc.~cit.~the basic ingredient is the direct image of the structural sheaf of $\PP^2$: $\varphi_*(\Oc_{\PP^2})$. Since $\varphi$ is assumed to be finite, then $\varphi_*(\Oc_{\PP^2})$ is a coherent sheaf and hence one can consider its Fitting ideal sheaves \cite{MP89}. In our approach, we did not considered this sheaf, but the sheafification of the graded part of degree $\nu$ of the symmetric algebra $\Sc_I$ ($\nu\geq \nu_0$). It turns out that they are isomorphic as soon as $\phi$ is finite. Let us be more precise. 

Assume that $\phi$ is finite (e.g.~$\phi$ has no base point). In the following diagram
$$
\xymatrix@1{ \Gamma \subset \PP^2_k\times\PP^3_k \ar[d]_{\pi_1} \ar[rd]^{\pi_2} &  \\     \PP^2_k \ar[r]^{\phi} & \PP^3_k}
$$
the morphism ${\pi_1}_{|\Gamma}:\Gamma \rightarrow \PP^2_k$ is an isomorphism, so that $ \pi_1^{\sharp}: \Oc_{\PP^2}\simeq {\pi_1}_*({\Oc_{\Gamma}})$ by definition. In particular, we have the sheaf isomorphisms
\begin{equation}\label{eq:phiPP1}
	{\pi_2}_{*}(\Oc_\Gamma) =  \phi_*{\pi_1}_*({\Oc_{\Gamma}}) \simeq \phi_*(\Oc_{\PP^2}).
\end{equation}
Given an integer $\nu$, we can consider two sheaves of $\Oc_{\PP^3}$-modules: 
$${\pi_2}_*(\Oc_\Gamma\otimes_{\Oc_{\PP^2\times \PP^3}} \pi_1^*(\Oc_{\PP^2}(\nu)))$$
and the sheaf associated to the graded $R$-module $(\Sc_I)_\nu$ that we will denote by $\widetilde{(\Sc_I)_\nu}$. 

\begin{lem} For all $\nu\geq \nu_0$ we have
$$ \phi_*(\Oc_{\PP^2}) \simeq \widetilde{(\Sc_I)_\nu}$$
\end{lem}
\begin{proof}
	The exact sequence of $R$-modules
	$$ 0 \rightarrow H^0_\mm(\Sc_I)_\nu \rightarrow (\Sc_I)_\nu \rightarrow H^0(\PP^2_{\Spec(R)},\Oc_\Gamma\otimes \pi_1^*(\Oc_{\PP^2}(\nu))) 
	\rightarrow H^1_\mm(\Sc_I)_\nu \rightarrow 0$$
	shows that for all $\nu\geq \nu_0$ we have (extending $\pi_2$ to $\PP^2\times \Spec(R) \rightarrow \Spec(R)$)
	\begin{equation}\label{eq:pi2pi1}
	(\Sc_I)_\nu= H^0(\PP^2_{\Spec(R)},\Oc_\Gamma\otimes \pi_1^*(\Oc_{\PP^2}(\nu)))=H^0(\Spec(R),{\pi_2}_{*}(\Oc_\Gamma\otimes \pi_1^*(\Oc_{\PP^2}(\nu))))
	\end{equation}
	where the last equality follows from the definition of ${\pi_2}_*(-)$ and the fact that $\pi_2^{-1}(\Spec(R))=\PP^2\times \Spec(R)$. Since \eqref{eq:pi2pi1} is an equality of graded $R$-modules, we deduce that 
	$$\widetilde{(\Sc_I)_\nu}\simeq {\pi_2}_*(\Oc_\Gamma\otimes_{\Oc_{\PP^2\times \PP^3}} \pi_1^*(\Oc_{\PP^2}(\nu))).$$
	Comparing this with \eqref{eq:phiPP1} we obtain that for all $\nu\geq \nu_0$
	$$ \phi_*(\Oc_{\PP^2}) \simeq \widetilde{(\Sc_I)_\nu}$$
	(because $\pi_1^*(\Oc_{\PP^2}(\nu))$ is an invertible sheaf).	
\end{proof}

In particular, we get that
$$\Fitt^0_{\PP^3}( \pi_1^*(\Oc_{\PP^2}(\nu))) \simeq \Fitt^0_{\PP^3}( \widetilde{(\Sc_I)_\nu} ) \simeq \widetilde{\Fitt^0_\nu(\phi)}$$
where the last isomorphism follows from the stability of Fitting ideals under base change and the classical properties of the sheaves associated to graded modules.

\section{The case of bi-graded parameterizations of surfaces}\label{sec:bigraded}

In this paper, we have treated rational surfaces that are parameterized by a birational map from $\PP^2$ to $\PP^3$. In the field of geometric modeling such surfaces are called \emph{rational triangular B\'ezier surfaces} (or patches). In this section, we will (briefly) show that we can treat similarly rational surfaces that are parameterized by a birational map from $\PP^1\times \PP^1$ to $\PP^3$; such surfaces are called \emph{rational tensor product surfaces} in the field of geometric modeling where they are widely used.

\medskip

Suppose given a rational map $\phi: \PP^1_k\times \PP^1_k \dto \PP^3_k$ where $f_0,f_1,f_2,f_3$ are four bi-homogeneous polynomials of bi-degree $\ud:=(d_1,d_2)$, that is to say that $f_i \in S:=k[s_0,s_1,t_0,t_1]$ is homogeneous of degree $d_1$ with respect to the variables $s_0,s_1$ and is homogeneous of degree $d_2$ with respect to the variables $t_0,t_1$.  As in Section \ref{sec:def}, denote by $\Scc$ the closure of the image of $\phi$, set $x=(x_0,x_1,x_2,x_3)$, $R:=k[x]$, $I=(f_0,f_1,f_2,f_3)\subset S_\ud$ and assume that conditions \eqref{H} hold. In this setting, the approximation complex of cycles can still be considered and it inherits a bi-graduation from the ring $S$ with respect to the two sets of variables $s_0,s_1$ and $t_0,t_1$. It is proved in \cite{Bot10} that matrix representations of $\phi$ still exist in this setting. To be more precise, define the following region (see Figure \ref{fig:R})
\begin{multline*}
\Rf_{\ud}:=\{ (n_1,n_2)\in \ZZ^2 \textrm{ such that } (n_1 \leq d_1-2), \ \textrm{ or } (n_2 \leq d_2-2), \textrm{ or } \\
  (n_1 \leq 2d_1-2 \textrm{ and } n_2 \leq 2d_2-2) \}\subset \ZZ^2.	
\end{multline*} 
\begin{figure}[!h]
	\centering
	\includegraphics[scale=1]{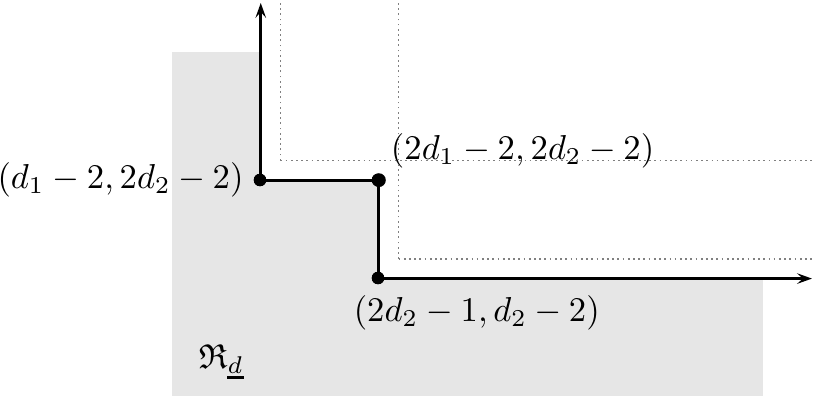}
	\caption{The region $\Rf_{\ud}$ is represented in grey color.}
	\label{fig:R}
\end{figure}

\noindent Then, for all $\unu=(\nu_1,\nu_2) \notin \Rf_\ud$ the matrix $M(\phi)_\unu$, which corresponds to the $\unu$-graded part of the matrix $M(\phi)$ in \eqref{eqZComplex}, is a matrix representation of $\phi$ with all the expected properties, in particular (see \cite{Bot10} for the details):
\begin{itemize}
	\item its entries are linear forms in $R=k[x_0,\ldots,x_3]$,
	\item it has $(\nu_1+1)(\nu_2+1)$ rows and at least as many columns as rows,
	\item it has maximal rank $(\nu_1+1)(\nu_2+1)$ at the generic point of $\PP^3_k$,
	\item when specializing $M(\phi)_\unu$ at a point $P\in \PP^3$, its rank drops if and only if $P\in \Sc$.
\end{itemize}
Therefore, one can also consider the Fitting ideals $\Fitt^i_\unu(\phi) \subset R$ associated to $M(\phi)_\unu$ for which Proposition \ref{propVanishingFitt} holds. Notice that in this setting, the projective plane $\PP^2$ is replaced by $\PP^1\times \PP^1$ in \eqref{eq:diagram}, so that for any point $\pp \in \Proj (R)$ the fiber of $\pi_2$ at $\pp$ is a subscheme of $\PP^1_{\kappa(\pp)}\times \PP^1_{\kappa(\pp)}$:
$$\pi_2^{-1}(\pp) = \mathrm{BiProj}(\Sc_I\otimes_R \kappa(\pp)) \subset  \PP^1_{\kappa(\pp)}\times \PP^1_{\kappa(\pp)}.$$
Thus, the Hilbert function of a fiber $\pi_2^{-1}(\pp)$ is a function of $\ZZ^2$. Similarly to Section \ref{sec:main}, we set $N^\pp_\unu:=HF_{\Sc_I\otimes_R \kappa(\pp)}(\unu)$. It turns out that an equality similar to \eqref{eq:HPHF} also holds:
$$HP_{\Sc_I\otimes_R \kappa(\pp)} (\unu ):=HF_{\Sc_I\otimes_R \kappa(\pp)}(\unu )-\sum_i (-1)^i HF_{H_{B}^i({\Sc_I\otimes_R \kappa(\pp)})}(\unu)$$
where the ideal $B$ denotes the product of ideals $B:=(s_0,s_1)(t_0,t_1)\subset S$. The function $HP_{\Sc_I\otimes_R \kappa(\pp)}$ is a polynomial function called the Hilbert polynomial (see \cite{ColTesis}  or \cite[Chapter 4.6]{BC13} for more details). The following result allows to control the vanishing of the Hilbert function of the local cohomology modules of the fibers.

\begin{thm} For all $\pp\in \Spec(R)$, all integer i and $\unu \notin \Rf_\ud$,
	$$H_{B}^i({\Sc_I\otimes_R \kappa(\pp)})_\unu =0$$
\end{thm}
\begin{proof} We will only outline the proof of this result because it follows the main lines of the proof of Corollary \ref{cor:key}.
	
	Assume first that  $\delta:=\dim(\Sc_I\otimes_R \kappa(\pp))=1$. Lemma \ref{lemFromDimToN} implies that 
	$$H_{B}^0({\Sc_I\otimes_R \kappa(\pp)})_\unu=H_{B}^1({\Sc_I\otimes_R \kappa(\pp)})_\unu=0$$ if $H_{B}^0({\Sc_I})_\unu=H_{B}^1({\Sc_I})_\unu=0$. As proved in \cite[Theorem 4.7]{Bot10}, this latter condition holds if $\unu \notin \Rf_\ud$.
	
	Now, we turn to the case $\delta=2$ for which we have to control the vanishing of $H_{B}^0({\Sc_I\otimes_R \kappa(\pp)})_\unu$, $H_{B}^1({\Sc_I\otimes_R \kappa(\pp)})_\unu$ and $H_{B}^2({\Sc_I\otimes_R \kappa(\pp)})_\unu$. We proceed as in the proof of Theorem \ref{thm:1dimfibers}.
We may assume for simplicity that $\pp=(1:0:0:0)$ so that $I=(f_0)+h(g_1,g_2,g_3)$ 
with $h$ the equation of the unmixed part of dimension 2 of the fiber at $\pp$; it is a bi-homogeneous polynomial of bi-degree $\ue=(e_1,e_2)$. Setting $I':=(g_1,g_2,g_3)$, we have an exact sequence
		$$ 
		 0 \rightarrow Z_1'(\ud) \rightarrow Z_1(\ud) \rightarrow S \rightarrow \Sc_I\otimes \kappa(\pp) \rightarrow 0,$$
		 with $Z_1'$ the syzygy module of $(f_1,f_2,f_3)$, the first map sending $(b_1,b_2,b_3)$ to
		 $(0,b_1,b_2,b_3)$ and the second  $(a_0,a_1,a_2,a_3)$ to $a_0$. The two spectral sequences associated to the double complex
		$$ 0 \rightarrow \Cc_B^\bullet(Z_1'(\ud)) \rightarrow \Cc_B^\bullet (Z_1(\ud)) \rightarrow \Cc_B^\bullet(S)$$
	show that
	\begin{itemize}
		\item $H_{B}^0({\Sc_I\otimes_R \kappa(\pp)})_\unu=0$ if $H^1_B(Z_1(\ud))_\unu=H^2_B(Z_1'(\ud))_\unu=0$,
		\item $H_{B}^1({\Sc_I\otimes_R \kappa(\pp)})_\unu=0$ if $H^2_B(Z_1(\ud))_\unu=H^3_B(Z_1'(\ud))_\unu=0$,
		\item $H_{B}^2({\Sc_I\otimes_R \kappa(\pp)})_\unu=0$ if $H^2_B(S)_\unu=H^3_B(Z_1(\ud))_\unu=0$.
	\end{itemize}
Recall that $H^i_B(S)=0$ if $i\neq 2,3$, that $H^3_B(S)_\unu\neq 0$ if and only if $\nu_1\leq -2$ and $\nu_2\leq -2$, and that $H^2_B(S)_\unu\neq 0$ if and only if $\nu_1\leq-2$ and $\nu_2\geq 0$, or $\nu_1\geq 0$ and $\nu_2\leq-2$ (see \cite[Lemma 6.7 and \S 7]{Bot10}). 
	
	To control the local cohomology of $Z_1$ we consider the long exact sequence of local cohomology associated to the canonical  short exact sequence 
$$ 0 \rightarrow Z_1 \rightarrow S(-\ud)^4 \rightarrow I \rightarrow 0.$$ We get $H^1_B(Z_1)=0$ and the exact sequences
$H^1_B(I)\rightarrow H^2_B(Z_1) \rightarrow H^2_B(S(-\ud)^4)$ and $H^2_B(I)\rightarrow H^3_B(Z_1) \rightarrow H^3_B(S(-\ud)^4)$. Since we know the local cohomology of $S$, it remains to estimate the vanishing of the local cohomology modules $H^1(I)$ and $H^2(I)$. For that purpose, we examine the two spectral sequences associated to the double complex
$$ \Cc_B^\bullet(K_\bullet((f_0,f_1,f_2,f_3);S)).$$
We obtain that $H^1_B(I)_\unu=0$ for all $\unu$ such that $H^2_B(S(-2\ud))_\unu=H^3_B(S(-3\ud))_\unu=0$ and that $H^2_B(I)_\unu=0$ for all $\unu$ such that $H^2_B(S(-\ud))_\unu=H^3_B(S(-2\ud))_\unu=0$. From here, we deduce that $H^2_B(Z_1(\ud))_\unu=H^3_B(Z_1(\ud))_\unu=0$ for all $\unu \notin \Rf_\ud$.

Similarly, to control the local cohomology of $Z_1'$ we consider the long exact sequence of local cohomology associated to the canonical  short exact sequence 
$$0 \rightarrow Z_1'(\ue) \rightarrow S(-\ud+\ue)^3 \rightarrow I' \rightarrow 0.$$
We get $H^1_B(Z_1')=0$ and the exact sequences
$H^1_B(I')\rightarrow H^2_B(Z_1'(\ue)) \rightarrow H^2_B(S(-\ud+\ue)^3)$ and $H^2_B(I')\rightarrow H^3_B(Z_1'(\ue)) \rightarrow H^3_B(S(-\ud+\ue)^3)$.
As in the previous case, it remains to estimate the vanishing of the local cohomology modules $H^1(I')$ and $H^2(I')$. We proceed similarly by inspecting the two spectral sequences associated to the double complex $\Cc_B^\bullet(K_\bullet((g_1,g_2,g_3);S))$ and finally deduce that $H^2_B(Z_1'(\ud))_\unu=H^3_B(Z_1'(\ud))_\unu=0$ for all $\unu \notin \Rf_\ud$.
\end{proof}

As a consequence of this result, all the properties we obtained in Section \ref{sec:comp} also applied to the setting of a bi-graded parameterization. In particular, we have the following properties:
\begin{itemize}
	\item If $\dim \pi_2^{-1}(\pp)\leq 0$ then for all $\unu\notin \Rf_\ud$ 
	$$\corank(M_\unu(\phi)(\pp))=\deg(\pi_2^{-1}(\pp)).$$
	\item If $\dim \pi_2^{-1}(\pp) = 1$ then for all $\unu\notin \Rf_\ud$ 
	$$\corank(M_\unu(\phi)(\pp))=e_1\nu_1+e_2\nu_2+c$$
	where the curve component of $\pi_2^{-1}(\pp)$ is of bi-degree $(e_1,e_2)$ and $c\in \ZZ$.
\end{itemize}
From here, an algorithm similar to the one presented in \S \ref{sec:algo} can be derived, for all $\unu\notin \Rf_\ud$, in order to determine the characters of a fiber $\pi_2^{-1}(\pp)$ by comparing matrix representations with consecutive indexes. Typically, the first matrix representation to consider is $M(\phi)_{(d_1-1,2d_2-1)}$ (or $M(\phi)_{(2d_1-1,d_2-1)}$) which has $2d_1d_2$ rows. Then, informations on the fiber at a given point $\pp$ can be obtained by comparing the rank at $\pp$ of $M(\phi)_{(d_1-1,2d_2-1)}$ with the ranks at $\pp$ of $M(\phi)_{(d_1,2d_2-1)}$ and $M(\phi)_{(d_1-1,2d_2)}$.

\subsection*{Acknowledgments}
We would like to thank Ragni Piene and Bernard Teissier for enlightening discussions about the singularities of finite maps. Most of this work was done while the first author was hosted at INRIA Sophia Antipolis with a financial support of the european Marie-Curie Initial Training Network SAGA (ShApes, Geometry, Algebra), FP7-PEOPLE contract PITN-GA-2008-214584.

\def\cprime{$'$}

\end{document}